\documentclass[12pt,reqno]{amsart}

\usepackage{amssymb, amsmath, amsthm, amscd, mathrsfs, dsfont}

\newtheorem{theorem}{Theorem}[section]
\newtheorem{proposition}[theorem]{Proposition}
\newtheorem{lemma}[theorem]{Lemma}
\newtheorem{corollary}[theorem]{Corollary}

\theoremstyle{definition}
\newtheorem{definition}[theorem]{Definition}
\newtheorem{remark}[theorem]{Remark}

\numberwithin{equation}{section}

\newcommand{\C}{{\mathbb{C}}}

\newcommand{\T}{{\mathbb{T}}}

\newcommand{\Pn}{{\mathbb{P}^{n-1}}}
\newcommand{\G}{{\mathfrak{G}}}
	
\begin{document}
	
\title[Vector bundles over LVMB manifolds]{Stability and holomorphic connections on vector 
bundles over LVMB manifolds}

\author[I. Biswas]{Indranil Biswas}

\address{School of Mathematics, Tata Institute of Fundamental
Research, Homi Bhabha Road, Mumbai 400005, India}

\email{indranil@math.tifr.res.in}

\author[S. Dumitrescu]{Sorin Dumitrescu}

\address{Universit\'e C\^ote d'Azur, CNRS, LJAD}

\email{dumitres@unice.fr}

\author[L. Meersseman]{Laurent Meersseman}

\address{Laboratoire Angevin de Recherche en Math\'ematiques, Universit\'{e} d'Angers, Universit\'e de Bretagne-Loire F-49045 Angers Cedex, France}

\email{laurent.meersseman@univ-angers.fr}



\date{}

\begin{abstract}
We characterize all LVMB manifolds $X$ such that the holomorphic tangent bundle 
$TX$ is spanned at the generic point by a family of global holomorphic vector fields, each of 
them having non-empty zero locus. We deduce that holomorphic connections on semi-stable 
holomorphic vector bundles over LVMB manifolds with this previous property are always flat.
\bigskip

\noindent
{\it R\'esum\'e.}\,\, {\bf Fibr\'es vectoriels holomorphes sur les vari\'et\'es LVMB.}\,\,
Nous caract\'erisons les vari\'et\'es LVMB qui ont la propri\'et\'e de positivit\'e $\mathcal P$   suivante :  le fibr\'e tangent holomorphe est engendr\'e au point g\'en\'erique par une famille de champs de vecteurs holomorphes (globalement d\'efinis)   $\{ v_i \}$, tel que chaque $v_i$ s'annule en au moins un point de $X$.
Nous en d\'eduisons que, sur les  vari\'et\'es LVMB avec la propri\'et\'e $\mathcal P$, les connexions holomorphes sur les fibr\'es vectoriels holomorphes semi-stables sont n\'ecessairement plates.
\end{abstract}

\maketitle

\section*{\bf{Version fran\c caise abr\'eg\'ee}}

Les vari\'et\'es LVMB \cite{meerssemanthesis, meersseman2000, Bosio, 
BosioActa} (voir aussi  \cite{PUV}) forment une classe de vari\'et\'es complexes compactes (en g\'en\'eral, non--K\"ahler)
qui g\'en\'eralise les vari\'et\'es toriques et aussi les vari\'et\'es de Hopf diagonales.  Les  vari\'et\'es LVMB sont des compactfications \'equivariantes lisses de groupes de Lie complexes ab\'eliens.

Cet article \'etudie la g\'eom\'etrie complexe des fibr\'es holomorphes au-dessus d'une vari\'et\'e LVMB.

Rappelons qu'une vari\'et\'e  LVMB  $X$ est biholomorphe \`a un  quotient d'un ouvert $V$ de l'espace projectif complexe $\Pn$ par une action holomorphe,  libre et propre du groupe de Lie complexe   $\C^m$, engendr\'e par  $m$ (avec $n>2m$)  champs de vecteurs lin\'eaires diagonaux  qui commutent.

L'ouvert  $V$ est le compl\'ementaire dans  $\Pn$ d'une union $E$ de   $k$ hyperplanes  de coordonn\'ees. Par cons\'equent, l'ouvert $V$ contient le tore  $\T\,:=\,(\C^*)^{n-1}$ et est  invariant par l'action naturelle de  $\T$ 
sur  $\Pn$. Le quotient de $\T$ par l'action de  $\C^m$ s'identifie \`a un groupe de Lie complexe ab\'elien $G$. Cela implique que $X$ est une compactification \'equivariante lisse (de dimension $\nu=n-m-1$) de $G$ \cite{meerssemanthesis, meersseman2000, Bosio, 
BosioActa}.

Le premier r\'esultat principal de l'article (Th\'eor\`eme \ref{thLVMB}) est une caract\'erisation de vari\'et\'es LVMB $X$ qui ont la propri\'et\'e de positivit\'e $\mathcal P$ suivante : le fibr\'e tangent holomorphe $TX$ est engendr\'e au point g\'en\'erique par une famille de champs de vecteurs holomorphes (globalement d\'efinis)   $(v_1,\hdots,v_{n-m-1})$, tel que chaque $v_i$ s'annule en au moins un point de $X$.
Le th\'eor\`eme \ref{thLVMB} affirme qu'une vari\'et\'e LVMB a la propri\'et\'e $\mathcal P$ si et seulement si $k \leq m+1$.

Le cas $k=m+1$  est particuli\`erement int\'eressant car les vari\'et\'es LVMB correspondantes poss\`edent des structures affines complexes (notez que la preuve du Lemme 12.1 faite dans \cite{BosioActa}
pour les vari\'et\'es LVM se g\'en\'eralise sans difficult\'e au cadre LVMB). Les vari\'et\'es LVMB  les plus simples avec $k=m+1$ sont les vari\'et\'es de Hopf diagonales, mais il y a \'egalement des exemples  dont la topologie peut-\^etre  compliqu\'ee (vari\'et\'es diff\'eomorphes \`a un produit entre $(\mathbb S^1)^m$ et une somme connexe de produits de sph\`eres, ainsi que des vari\'et\'es avec torsion arbitraire dans le sens du Th\'eor\`eme  14.1 et du Corollaire  14.3 dans \cite{BosioActa}). 

Dans la deuxi\`eme partie  de l'article nous \'etudions les fibr\'es holomorphes $E$ munis de connexions holomorphes au-dessus des vari\'et\'es LVMB qui ont la propri\'et\'e $\mathcal P$. Plus pr\'ecis\'ement, nous faisons des liens avec la notion  de stabilit\'e au sens des pentes. Mentionnons que les notions de degr\'e, pente et (semi)-stabilit\'e  sont consid\'er\'ees ici par rapport \`a une m\'etrique de Gauduchon. En particulier, le degr\'e n'est pas, en g\'en\'eral, un invariant topologique (comme c'est le cas dans le contexte K\"ahler). 

Nous montrons que chaque sous-faisceau $E_i$ dans la filtration de Harder-Narasimhan 
de $E$

$$
0\,=\, E_0\, \subset\, E_1 \,\subsetneq\, E_2\, \subsetneq\, \cdots \, \subsetneq\, E_b \,
\subsetneq\, E_{b+1}\,=\, E
$$

 est invariant par la connexion holomorphe et est, par cons\'equent, localement libre. Nous utilisons ensuite la propri\'et\'e de positivit\'e $\mathcal P$ de $TX$ pour en d\'eduire que  la connexion holomorphe induite sur chaque
quotient semi-simple $E_i/E_{i-1}$    est n\'ecessairement plate. Par cons\'equent, nous obtenons le Corollaire \ref{cor2}  qui affirme que toute connexion holomorphe sur un fibr\'e semi-stable $E$ au-dessus d'une vari\'et\'e LVMB avec la propri\'et\'e $\mathcal P$ est n\'ecessairement  plate.

\section{Introduction}

An important class of compact complex (in general, non--K\"ahler) manifolds generalizing toric 
varieties and diagonal Hopf manifolds are the so-called LVMB manifolds \cite{meerssemanthesis, meersseman2000, Bosio, 
BosioActa} (see also \cite{PUV}). They are smooth equivariant compactifications of complex 
abelian Lie groups.

This article deals with the geometry of holomorphic vector bundles over LVMB manifolds.

Our first result (Theorem \ref{thLVMB}) is a geometric characterization of LVMB manifolds $X$ 
admitting a family of global holomorphic vector fields spanning the holomorphic tangent bundle 
$TX$ at the generic point and such that each vector field in the family has a non-empty zero 
locus.

We use this positivity result and stability arguments to show that holomorphic connections in 
semi-stable holomorphic vector bundles over those LVMB manifolds are always flat (see 
Corollary \ref{cor2}).

\section{Geometry of the tangent bundle of LVMB manifolds}

We refer to \cite{meersseman2000} and \cite{Bosio} for more details on the properties of LVMB manifolds 
used below. It should be clarified that although the article \cite{meersseman2000} deals with LVM 
manifolds, most results there hold verbatim, and with the same proof, for the more general LVMB case.

A LVMB manifold $X$ is a a quotient of some open and dense subset $V$ of $\Pn$ by a free and 
proper holomorphic $\C^m$-action generated by $m$ commuting linear diagonal vector fields 
$\xi_1,\hdots,\xi_m$ (with $n>2m$). More precisely, the action is given by
\begin{equation}
\label{completeaction}
T\cdot [z]:=\left [ 
z_ie^{\langle \Lambda_i,T\rangle}
\right ]_{i=1} ^n
\end{equation}
where the $m$ rows of the matrix $(\Lambda_1,\hdots,\Lambda_n)$ are the coefficients of the $m$ vector fields 
$\xi_1,\hdots,\xi_m$  as linear combinations of $z_1\frac{\partial}{\partial z_1},\,\cdots,\,
z_n\frac{\partial}{\partial z_n}$.

The subset $V$ is the complement in $\Pn$ of a union $E$ of coordinate subspaces. It follows 
that it contains a torus $\T\,:=\,(\C^*)^{n-1}$ and is preserved under the natural action of $\T$ 
onto $\Pn$. The quotient of $\T$ by the $\C^m$-action is a complex abelian Lie group, say $G$, 
and $X$ is a smooth equivariant compactification of $G$.

The following alternative description of $G$ can be found in \cite[p.27]{meerssemanthesis} (and which itself is a straightforward generalization of Theorem 1 in \cite{meersseman2000}).

\begin{proposition}
	\label{Glattice}
	Assume that 
	\begin{equation}
	\label{firstmcondition}
	\text{\rm rank}_{\mathbb C}
	\begin{pmatrix}
	\Lambda_1 &\hdots &\Lambda_{m+1}\cr
	1 &\hdots &1
	\end{pmatrix}
	=m+1.
	\end{equation}
	Then $G_\Lambda$ is isomorphic to the quotient of $\mathbb C^{n-m-1}$ by the $\mathbb Z^{n-1}$ abelian subgroup generated by
	$(Id, B_\Lambda A_\Lambda^{-1})$ where
	\begin{equation*}
	A_\Lambda=^t\kern-4pt(\Lambda_2-\Lambda_1,\hdots,\Lambda_{m+1}-\Lambda_1)
	\end{equation*}
	and
	\begin{equation*}
	B_\Lambda= ^t\kern-4pt(\Lambda_{m+2}-\Lambda_1,\hdots,\Lambda_{n-1}-\Lambda_1).
	\end{equation*}
\end{proposition}

\begin{remark}
	It is easy to prove that 
	\begin{equation*}
	\text{\rm rank}_{\mathbb C}
	\begin{pmatrix}
	\Lambda_1 &\hdots &\Lambda_n\cr
	1 &\hdots &1
	\end{pmatrix}
	=m+1.
	\end{equation*}
	(cf. {\rm \cite[Lemma 1.1]{meersseman2004holomorphic}} in the LVM case). Hence, up to a permutation, condition \eqref{firstmcondition} is always fulfilled.
\end{remark}
	
Let $\G$ be the Lie algebra of $G$. Given $a\in\G$, let 
$$
v(x)\,:=\,{\dfrac{d}{dt}}\left(x\cdot \exp(ta)\right)\vert_{t=0}\, ,\ \ x\,\in\, X
$$
be the corresponding fundamental vector field. 

\begin{definition}\label{defNVP}
We say that $X$ has the non-zero vanishing property if we may find a basis 
$(a_1,\hdots,a_{n-m-1})$ of $\G$ such that the corresponding fundamental vector fields 
$(v_1,\hdots,v_{n-m-1})$ all have a non-empty vanishing locus. \end{definition}

Let $k$ denote the number of coordinate hyperplanes in $E$. We have 

\begin{theorem}\label{thLVMB} A LVMB manifold $X$ has the non-zero vanishing property if and 
only if the inequality $k\,\leq\, m+1$ is fulfilled.
\end{theorem}

\begin{proof}
Observe that the natural map $V\to X$ induces a Lie group morphism $\T\to G$. Hence, using homogeneous coordinates $[z_1;\hdots,z_n]$ in $\mathbb P^{n-1}$, the image 
of the vector fields $z_1\frac{\partial}{\partial z_1},\,\cdots,\,
z_n\frac{\partial}{\partial z_n}$ form a generating family of fundamental vector fields of $X$.
	
Observe also that such a vector field $z_i\frac{\partial}{\partial z_i}$ vanishes
on some point of $X$ if and only if (the projectivization of) $\{z_i\,=\,0\}$ is not included in $E$.

Assume firstly that $k\,\leq\, m+1$. For simplicity, assume that the coordinate hyperplanes 
belonging to $E$ are $\{z_{n-k-1}\,=\,0\},\,\cdots ,\,\{z_n\,=\,0\}$. Observe that the $\mathbb C^m$-action \eqref{completeaction} preserves the zero and the non-zero coordinates. Given $[z]\,\in\, V$,
since $k\,\leq\, m+1$, 
we may thus find some $T\,\in\,\C^m$ such that the image $w$ of $z$ under the action of $T$ has 
the form
$$
[w_1,\,\hdots, \,w_{n-k-2},\,1,\,\hdots,\,1]\, .
$$
Such a $T$ is not unique but unique up to some additive subgroup $H$ of $\C^m$ of dimension $m-k-1$,
cf. the statement of Proposition \ref{Glattice} whose proof follows the same line of arguments. In other
words, setting 
$$
V_0\,:=\,\{[w_1,\,\hdots,\, w_{n-k-2},\,1,\hdots,\,1]\,\in \, V\}
$$
then $X$ is isomorphic to $V_0$ quotiented by the action of $H$. As a consequence, the
smaller family $z_1\frac{\partial}{\partial z_1},\, \cdots,\,
z_{n-k-2}\frac{\partial}{\partial z_{n-k-2}}$ still descends as a generating
family of fundamental vector fields of $X$. Hence we may find $n-m-1$ vector fields
in this family which form a basis of fundamental vector fields. Now they all vanish on $X$
by the above observation.
	
To prove the converse, assume that $k\,>\,m+1$. Then in the generating family of
fundamental vector fields, we have strictly less than $n-m-1$ vector fields with non-empty
vanishing locus. Moreover, a linear combination of vector fields of the family, say
$$
\sum_{i=1}^n \alpha_iz_i\frac{\partial}{\partial z_i}
$$
has vanishing locus
$$
V\cap{\mathbb P} (\{z_i\,=\,0\,\mid\, \alpha_i\not \,=\,0\})
$$
Hence to have non-empty vanishing locus, a fundamental vector field must be a linear
combination of the $n-k$ vector fields with non-empty vanishing locus of the generating
family. Since $n-k\,<\,n-m-1$, we cannot find a basis of such vector fields. 
\end{proof}
	
\begin{remark}\label{rkaffine}
The case $k\,=\,m+1$ in Theorem \ref{thLVMB} is of special interest, because the corresponding LVMB 
manifolds admit an affine atlas, see Lemma 12.1 of \cite{BosioActa} for a proof in the LVM case, 
which trivially generalizes to LVMB manifolds.  Examples of LVMB fulfilling $k=m+1$ include diagonal Hopf manifolds, but also manifolds diffeomorphic to the product of $(\mathbb S^1)^m$ with connected sums of sphere products and even manifolds with arbitrary torsion in the sense of 
Theorem 14.1 and Corollary 14.3 of \cite{BosioActa}. 
\end{remark}

\section{Stability and holomorphic connections on LVMB manifolds}

As in Theorem \ref{thLVMB}, let $X$ be a LVMB manifold, of complex dimension $\nu\,:=\, n-m-1$,
such that $k\,\leq\, m+1$. A Gauduchon metric on $X$ is a Hermitian structure $h$ such that
the corresponding $(1,\, 1)$-form $\omega_h$ satisfies the equation $\partial\overline{\partial}
\omega^{\nu-1}_h\,=\, 0$; Gauduchon metrics exist \cite{Ga}. Fix a Gauduchon form
$\omega_h$ on $X$. The degree of a torsionfree coherent analytic sheaf $F$ on $X$ is defined to be
$$
\text{degree}(F)\,:=\,
\frac{\sqrt{-1}}{2\pi}\int_X K(\det F)\wedge \omega_h^{\nu-1}\, \in\, {\mathbb R}\, ,
$$
where $\det F$ is the determinant line
bundle for $F$ \cite[Ch.~V, \S~6]{Ko} (or Definition 1.34 in \cite{Br}) and $K$ is the
curvature for a hermitian connection on
$\det F$ compatible with the Dolbeault operator
$\overline \partial_{\det F}$. This degree is independent of the Hermitian
metric on $\det F$, because any two such
curvature forms differ by a $\partial \overline\partial$-exact $2$--form on $X$:
$$
\int_X (\partial \overline\partial u)\wedge \omega_h^{\nu-1}\,=\,
-\int_X u \wedge\partial \overline\partial \omega_h^{\nu-1}\,=\, 0\, .
$$

A torsionfree coherent analytic sheaf $F$ on $X$ is called \textit{semistable} if
$$
\frac{\text{degree}(V)}{\text{rank}(V)}\, \leq\, \frac{\text{degree}(F)}{\text{rank}(F)}
$$
for all nonzero coherent analytic subsheaf $V\, \subset\, F$ (see \cite[p.~44, Definition 1.4.3]{LT}, \cite[Ch.~V, \S~7]{Ko}).

Let
\begin{equation}\label{et1}
T_0 \,\subsetneq\, T_1\, \subsetneq\, \cdots \, \subsetneq\, T_{\ell} \,
\subsetneq\, T_{\ell+1}\,=\, TX
\end{equation}
be the Harder--Narasimhan filtration of the holomorphic tangent bundle $TX$. We note that
$TX$ is semistable if and only if $\ell\,=\, 0$.

\begin{lemma}\label{lem1}
The subsheaf $T_\ell$ in \eqref{et1} satisfies the following condition:
$$
{\rm degree}(TX/T_\ell)\, >\, 0\, .
$$
\end{lemma}

\begin{proof}
Fix a basis $(a_1,\,\hdots,\, a_{\nu})$ of $\G$ such that the corresponding fundamental vector fields
$(v_1,\,\hdots,\, v_{\nu})$ have the property that each of them has a non-empty vanishing locus.
Theorem \ref{thLVMB} ensures that such a basis exists. Let $r$ be the rank of
$TX/T_\ell$. Fix $r$ elements from $(v_1,\,\hdots,\, v_{\nu})$ satisfying the following condition:
their projections to $TX/T_\ell$ together generate $TX/T_\ell$ over a Zariski open subset. Let
$(w_1,\,\hdots,\, w_r)$ denote this subset of $(v_1,\,\hdots,\, v_{\nu})$. Therefore,
$w_1\wedge \cdots\wedge w_r$ is a holomorphic section of $\bigwedge^r TX$ which is not identically
zero.

Consider the holomorphic line bundle $\det (TX/T_\ell)$ over $X$. The quotient map
$TX\, \longrightarrow\, TX/T_\ell$ produces a natural projection
$q\, :\, \bigwedge^r TX\, \longrightarrow\, \det (TX/T_\ell)$. Let
$$
\sigma\,:=\, q(w_1\wedge \cdots\wedge w_r) \, \in\, H^0(X,\, \det (TX/T_\ell))
$$
be the nonzero holomorphic section given by $w_1\wedge \cdots\wedge w_r$. If any
of $(w_1,\,\hdots,\, w_r)$ vanishes at a point $x\, \in\, X$, then clearly we have
$\sigma(x)\,=\, 0$. In particular, the divisor ${\rm div}(\sigma)$ is nonzero. Consequently,
we have (see \cite[Proposition 5.23]{Br})
$$
{\rm degree}(TX/T_\ell)\,=\, \text{Volume}_h({\rm div}(\sigma)) \,>\, 0\, .
$$
This completes the proof.
\end{proof}

Let $E$ be a holomorphic
vector bundle on $X$ equipped with a holomorphic connection $D$ (see \cite{At} for
holomorphic connections). Let
\begin{equation}\label{e21}
0\,=\, E_0\, \subset\, E_1 \,\subsetneq\, E_2\, \subsetneq\, \cdots \, \subsetneq\, E_b \,
\subsetneq\, E_{b+1}\,=\, E
\end{equation}
be the Harder--Narasimhan filtration of $E$, so $E$ is semistable if and only if $b\,=\, 0$.

\begin{proposition}\label{prop1}
For every $1\, \leq\, i\, \leq\, b-1$, the subsheaf $E_i\, \subset\, E$ in \eqref{e21}
is preserved by the holomorphic connection $D$.
\end{proposition}

\begin{proof}
Take any $1\, \leq\, i\, \leq\, b-1$. Let 
$$
S_i\, :\, TX\otimes E_i\, \longrightarrow\, E/E_i
$$
be the homomorphism given by the second fundamental form of $E_i$ for the connection $D$. We recall
that this second fundamental form is given by the composition
$$
E_i\, \stackrel{D}{\longrightarrow}\, E\otimes \Omega^1_X \,
\stackrel{q_i\otimes{\rm Id}}{\longrightarrow}\, (E/E_i)\otimes \Omega^1_X\, ,
$$
where $q_i\, :\, E\, \longrightarrow\, E/E_i$ is the quotient map. Tensoring both sides of
it with $TX$ and taking trace, the above homomorphism $S_i$ is obtained from the
second fundamental form.

For any holomorphic vector field $v$ on $X$, consider the composition
\begin{equation}\label{e3}
E_i\, \stackrel{v\otimes -}{\longrightarrow}\, TX\otimes E_i \, \stackrel{S_i}{\longrightarrow}\,
E/E_i\, .
\end{equation}
From the properties of the Harder--Narasimhan filtration we know that there is no nonzero
homomorphism from $E_i$ to $E/E_i$. Hence the composition homomorphism in \eqref{e3}
vanishes identically. Since global holomorphic vector fields generate $TX$ over a Zariski open
subset, the proposition follows.
\end{proof}

\begin{corollary}\label{cor1}
Each subsheaf $E_i$ in \eqref{e21} is a holomorphic subbundle of $E$.
\end{corollary}

\begin{proof}
By Lemma 4.5 in \cite{BD}, any coherent analytic subsheaf $E_i$ invariant by
the holomorphic connection $D$ is locally free.
\end{proof}

For each $1\, \leq\, i\, \leq\, b$, let $D_i$ be the holomorphic connection, induced by
$D$, on the quotient bundle $E_i/E_{i-1}$ in \eqref{e21}.

\begin{proposition}\label{prop2}
The holomorphic connection $D_i$ on $E_i/E_{i-1}$ is flat for every $i$.
\end{proposition}

\begin{proof}
A holomorphic vector bundle $V$ on $X$ is semistable if and only if $V$ admits
an approximate Hermitian--Einstein structure \cite[p.~629, Theorem 1.1]{NZ} (set
the Higgs field $\phi$ in \cite[Theorem 1.1]{NZ} to be zero). If
$V_1$ and $V_2$ are semistable holomorphic vector bundles on $X$, then
approximate Hermitian--Einstein structures on $V_1$ and $V_2$ together produce
an approximate Hermitian--Einstein structure on $V_1\otimes V_2$. Hence
$V_1\otimes V_2$ is also semistable.

Since $E_i/E_{i-1}$ is semistable, from the above observation we conclude that
${\rm End}(E_i/E_{i-1})$ is also semistable. Note that we have
$$\text{degree}({\rm End}(E_i/E_{i-1}))\,=\, 0\, ,$$
because $\det {\rm End}(E_i/E_{i-1})\,=\, {\mathcal O}_X$.

Let
$$
{\mathcal K}(D_i)\, \in\, H^0(X,\, \Omega^2_X\otimes {\rm End}(E_i/E_{i-1}))
$$
be the curvature of the holomorphic connection $D_i$ on $E_i/E_{i-1}$. For any
holomorphic vector field $v$ on $X$, the contraction
$i_v {\mathcal K}(D_i)$ of this curvature form gives a holomorphic homomorphism
$$
\widetilde{i_v {\mathcal K}(D_i)}\, :\, TX\, \longrightarrow\, {\rm End}(E_i/E_{i-1})\, .
$$
Now, since ${\rm End}(E_i/E_{i-1})$ i semistable of degree zero, from Lemma \ref{lem1}
it follows immediately that there is no nonzero holomorphic homomorphism from
$TX$ to ${\rm End}(E_i/E_{i-1})$. Hence the above homomorphism
$\widetilde{i_v {\mathcal K}(D_i)}$ vanishes identically.
Since global holomorphic vector fields generate $TX$ over a Zariski open
subset, we conclude that ${\mathcal K}(D_i)\,=\, 0$.
\end{proof}

The following is an immediate consequence of Proposition \ref{prop2}.

\begin{corollary}\label{cor2}
Let $E$ be a semistable holomorphic vector bundle on $X$, and let $D$ be a holomorphic
connection on $E$. Then $D$ is flat.
\end{corollary}

\end{document}